\documentclass[11pt,a4paper]{amsart}
\usepackage{amscd,amssymb,amsopn,amsmath,amsthm,graphics,amsfonts,enumerate,verbatim,calc}
\usepackage[dvips]{graphicx}
\input xy
\xyoption{all}

\usepackage{amssymb,amsmath}
\usepackage{mathpazo}

\headsep=1cm \numberwithin{equation}{section}
\newtheorem{theorem}{Theorem}[section]

\newtheorem{lemma}[theorem]{Lemma}
\newtheorem{proposition}[theorem]{Proposition}
\newtheorem{corollary}[theorem]{Corollary}

\theoremstyle{definition}
\newtheorem{definition}[theorem]{Definition}
\theoremstyle{remark}
\newtheorem{remark}[theorem]{Remark}

\newtheorem{example}[theorem]{Example}
\newtheorem{problem}[theorem]{Problem}
\newtheorem{fact}[theorem]{Fact}
\newtheorem{notation}[theorem]{Notation}

\newtheorem{discussion}[theorem]{Discussion}
\newtheorem{acknowledgement}{Acknowledgement}

\newcommand{\Ass}{\operatorname{Ass}}
\newcommand{\im}{\operatorname{im}}
\newcommand{\Ker}{\operatorname{ker}}
\newcommand{\grade}{\operatorname{grade}}

\newcommand{\HH}{\operatorname{H}}

\newcommand{\Ht}{\operatorname{ht}}

\newcommand{\Ext}{\operatorname{Ext}}

\newcommand{\Hom}{\operatorname{Hom}}

\newcommand{\End}{\operatorname{end}}
\newcommand{\Beg}{\operatorname{beg}}

\newcommand{\reg}{\operatorname{reg}}
\newcommand{\depth}{\operatorname{depth}}

\newcommand{\mm}{\mathfrak m}
\newcommand{\pp}{\mathfrak p}
\newcommand{\Min}{\operatorname{Min}}

\newcommand{\lo}{\longrightarrow}
\newcommand{\fm}{\frak{m}}
\newcommand{\fp}{\frak{p}}

\begin{document}
	
	\author[]{Mohsen Asgharzadeh}
	
	\email{mohsenasgharzadeh@gmail.com}
	
	\title[ ]{Huneke's Degree-Computing Problem}
	
	\dedicatory{}
	\begin{abstract}
		Let $I$ be a homogeneous ideal in a polynomial ring over a field, and let
		$I^{(n)}$ denote its $n$-th symbolic power. Huneke asked whether, if $I$ is
		prime and generated in degrees at most $D$, the symbolic power $I^{(n)}$ is
		generated in degrees at most $Dn$. We present several observations in support
		of this problem and discuss its limitations. For ideals in three variables we
		show that a linear bound always holds asymptotically  we obtain the exact bound $Dn$. On the other hand, we
		present explicit radical ideals in six and seven variables, generated in
		degree four, whose second symbolic power requires a generator of degree nine,
		thus violating Huneke's conjectural bound. We also prove a polynomial (rather than linear)
		bound for arbitrary homogeneous ideals, and an asymptotically sharp estimate
		for primes that are locally complete intersections away from the maximal
		ideal. 
	\end{abstract}
	\maketitle
	
	\smallskip
	\section{Introduction}
	
Let $A_m:=K[x_1,\ldots,x_m]$ be the polynomial ring in $m$ variables over a
field $K$. We drop the subscript $m$ when there is no risk of confusion, and
write $A:=A_m$. Let $I$ be an ideal of $A$. The $n$-th \textit{symbolic power}
of $I$ is
\[
I^{(n)}:=\bigcap_{\fp\in \Ass(I)}\bigl(I^nA_{\fp}\cap A\bigr).
\]
Symbolic powers encode the local behavior of ordinary powers at the associated
primes of $I$, and comparing the two is a classical and still very active
theme in commutative algebra. Huneke \cite{h1} posted the following problem.

\begin{problem}\label{1.1}(Understanding symbolic powers).
	Let $\fp\lhd A$ be a homogeneous prime ideal generated in degrees $\leq D$.
	Is $\fp^{(n)}$ generated in degrees $\leq Dn$?
\end{problem}

The problem is elementary when $m<3$. The purpose of this note is to collect
several observations, both positive and negative, that clarify the scope of
Huneke's question. Our first observation concerns rings of dimension three.

\smallskip
\noindent\textbf{Observation A.}
Let $I\lhd A_3$ be a radical ideal generated in degrees $\leq D$. Then
$I^{(n)}$ is generated in degrees $<(D+1)n$ for all $n\gg0$.

The point is to connect Problem \ref{1.1} to the cohomology functor
$\HH^0_{\fm}(-)$, which unifies our interest in symbolic powers with the
\textit{$(LC)$ property} from tight closure theory, introduced by Hochster and
Huneke. As a consequence we obtain an unconditional linear bound in dimension
three.

\begin{corollary}
	Let $I\lhd A_3$ be any radical ideal. There is $D\in \mathbb{N}$ such that
	$I^{(n)}$ is generated in degrees $\leq Dn$ for all $n>0$.
\end{corollary}

Our second observation gives the exact bound for ideals of linear type in four
variables.

\smallskip
\noindent\textbf{Observation B.}
Let $I\lhd A_4$ be an ideal of linear type generated in degrees $\leq D$.
Then $I^{(n)}$ is generated in degrees $\leq Dn$.

The positive results above are in sharp contrast with the following example,
which was suggested to me by Hop D. Nguyen.

\begin{example}\label{introduction}
	There is a radical ideal $I\lhd A_6$ generated in degrees $\leq 4$ such that
	$I^{(2)}$ is not generated in degrees $\leq 2\cdot 4=8$.
\end{example}

A simpler example exists over $A_7$; via the flat extension
$A_6\to A_n$ one produces counterexamples over $A_n$ for all $n>5$. Thus
Huneke's bound fails in general, already for the second symbolic power. These
counterexamples are proved in detail in Section 4.
For a nonzero proper
homogeneous ideal $J$, denote by
\(
d(J)=\max\{j\in\mathbb Z:(J/\mm J)_j\neq0\}
\)
the maximal degree of a minimal homogeneous generator of $J$.
Section 5 is devoted to an exact computation for a natural family of
squarefree monomial ideals. For $r\geq1$ set
\[
S_r=K[a,b,c,a_1,\ldots,a_r,b_1,\ldots,b_r,c_1,\ldots,c_r],
\]
write $A=\prod_i a_i$, $B=\prod_i b_i$, $C=\prod_i c_i$, and put
\[
I_r=(abc,\;bcA,\;acB,\;abC),\qquad F_r=abcABC.
\]
This is the edge ideal of the graph consisting of a triangle on $a,b,c$ with
$r$ leaves attached to each vertex, as in \cite[Example~6.5]{HaTrung}. We slightly extend this, by showing $d(I_r^{(2)})=3r+3$, so that
$d(I_r^{(2)})-2d(I_r)=r-1>0$ for $r\geq2$.

Our next observation concerns the general asymptotic behavior.

\smallskip
\noindent\textbf{Observation C.}
Let $I$ be any ideal whose symbolic Rees algebra is finitely generated. There
is $E\in \mathbb{N}$ such that $I^{(n)}$ is generated in degrees $\leq En$ for
all $n>0$.

For monomial ideals, Observation C follows from \cite[Theorem 2.10]{her3}
together with well-known facts; in this special case we can determine $E$
explicitly.

\begin{corollary}
	Let $I$ be a monomial ideal and let $f$ be the least common multiple of the
	generating monomials of $I$. Then $I^{(n)}$ is generated in degrees
	$\leq \deg(f)n$ for all $n>0$.
\end{corollary}

We derive this sharp bound for monomial radical ideals both from the corollary
above and by an elementary method. There are many ideals $I$ for which the
symbolic Rees algebra is not finitely generated but for which a linear bound
still holds. Indeed, Roberts \cite{p} constructed a prime ideal $\fp$ over
$A_3$ whose symbolic Rees algebra is not finitely generated; nevertheless,
Corollary 1.2 shows that there is $D$ with $\fp^{(n)}$ generated in degrees
$\leq Dn$ for all $n>0$. These considerations suggest the following question.

\begin{problem}
	Let $I\lhd A$ be any ideal. Is there a polynomial $f\in \mathbb{Q}[X]$ such
	that $I^{(n)}$ is generated in degrees $\leq f(n)$ for all $n>0$? Can $f$ be
	taken linear?
\end{problem}

Section 5 addresses the polynomial-bound part of the question above. The main
new ingredient is a uniform one-element saturation lemma: for any nonzero
homogeneous ideal $I$ there is a single nonzero polynomial $h$ such that
\[
I^{(n)}=I^n:h^\infty=\bigcup_{e\geq0}(I^n:h^e)
\qquad\text{for every }n\geq1,
\]
and the same holds if symbolic powers are defined with respect to the minimal
primes of $I$. The proof combines prime avoidance with Brodmann's theorem on
the finiteness of $\bigcup_n\Ass(A/I^n)$. Combining this saturation with the
standard elimination identity
$I^n:h^\infty=(I^n,1-Th)A[T]\cap A$ and with Dub\'e's universal
Gr\"obner-basis degree bound, we obtain the following.

\smallskip
\noindent\textbf{Observation D.}
	Let $I\subsetneq K[x_1,\ldots,x_m]$ be a nonzero homogeneous ideal. There
	is an integer $C_I\geq1$ such that, for every $n\geq1$,
$
	d(I^{(n)})\leq 2 (\frac{C_I^2n^2}{2}+C_In )^{2^m}.
$
	In particular, $d(I^{(n)})$ is bounded by a polynomial in $n$ of degree
	$2^{m+1}$.

This settles the polynomial part of the question above, but the bound is
extremely coarse and does not give linear growth. In particular it does not
settle Huneke's problem for arbitrary homogeneous ideals.

Section 7 gives an asymptotically sharp estimate for a class of primes. We
show that if $P\neq\mm$ is a homogeneous prime ideal generated in degrees at
most $D$ and $P$ is locally a complete intersection at every prime other than
$\mm$, then there is an integer $B_P\geq0$ with
\[
d(P^{(n)})\leq Dn+B_P\qquad(n\geq1).
\]
Consequently
$
\limsup_{n\to\infty}\frac{d(P^{(n)})}{n}\leq D.
$
The proof uses that, under this hypothesis, $P^n$ has no embedded associated
primes except possibly $\mm$, hence $P^{(n)}=(P^n:\mm^\infty)$, and that
saturation with respect to $\mm$ does not increase Castelnuovo--Mumford
regularity. The asymptotic regularity theorem of Cutkosky--Herzog--Trung and
Kodiyalam then gives the claim. The result applies, for instance, when
$\operatorname{Proj}(S/P)$ is smooth over $K$, but it does not recover the
exact bound $d(P^{(n)})\leq Dn$; the additive constant $B_P$ cannot simply be
discarded.

  We refer the reader to \cite{BH} for all
unexplained definitions in the sequel.
	\section{Preliminaries}
	
	We give a quick review of the material that we need. Let $R$ be any commutative
	ring with an ideal $\frak a$ with a generating set
	$\underline{a}:=a_{1}, \ldots, a_{r}$. By $\HH_{\underline{a}}^{i}(M)$ we mean
	the $i$-th cohomology of the \textit{$\check{C}$ech} complex of a module $M$
	with respect to $\underline{a}$. This is independent of the choice of the
	generating set. For simplicity, we denote it by $\HH_{\frak a}^{i}(M)$. We
	equip the polynomial ring $A$ with the standard graded structure. Then we can
	use the machinery of graded $\check{C}$ech cohomology modules.
	
	\begin{notation}
		Let $L=\bigoplus_{n\in \mathbb{Z}} L_n$ be a graded $A$-module and let
		$d\in \mathbb{Z}$.
		\begin{enumerate}
			\item[i)] The notation $L(d)$ refers to the $d$-th twist of $L$, i.e., shifting
			the grading $d$ steps.
			\item[ii)] The notation $\End(L)$ stands for $\sup\{n:L_n\neq 0\}.$
			\item[iii)] The notation $\Beg(L)$ stands for $\inf\{n:L_n\neq 0\}.$
		\end{enumerate}
	\end{notation}
	
	\begin{discussion}
		Denote the irrelevant ideal $\bigoplus_{n>0}A_n$ of $A$ by $\fm$.
		\begin{enumerate}
			\item[i)] We use the conventions $\sup\{\emptyset\}=-\infty$,
			$\inf\{\emptyset\}=+\infty$ and $-\infty< +\infty$.
			\item[ii)] Let $M$ be a graded $A$-module. Then $\HH_{\frak m}^{i}(M)$ carries
			a $\mathbb{Z}$-graded structure and $\End (\HH^i_{\fm}(M))<+\infty$.
		\end{enumerate}
	\end{discussion}
	
	\begin{definition}
		The \textit{Castelnuovo--Mumford} regularity of $M$ is
		$$\reg(M):=\sup\{\End(\HH^i_{\fm}(M))+i:0\leq i \leq \dim M\}.$$
	\end{definition}
	
	The $\reg(M)$ computes the degrees of generators in the following sense.
	
	\begin{fact}\label{f}
		A graded module $M$ can be generated by homogeneous elements of degrees not
		exceeding $\reg(M)$.
	\end{fact}
	
	The following easy fact translates a problem from symbolic powers to a problem
	on $\check{C}$ech cohomology modules.
	
	\begin{fact}
		Let $I\lhd A$ be a radical ideal of dimension one. Then
		$I^{(n)}/ I^n=\HH^0_{\fm}(A/I^n)$.
	\end{fact}
	
	We will use the following results.
	
	\begin{lemma}\label{1}(See \cite{C})
		Let $I\lhd A$ be a homogeneous ideal such that $\dim A/I\leq1$. Then
		$\reg (I^n) \leq n\reg(I)$ for all $n$.
	\end{lemma}
	
	Also, see \cite{ger}. Let us recall the following result from \cite{Cu} and
	\cite{K}. The regularity of $\reg I^{(n)}$ is equal to $dn + e$ for all large
	enough $n$. Here $d$ is the smallest integer $n$ such that
	$$(x : x \in I,\textit{ and x is homogeneous of degree at most n})$$ is a
	reduction of $I$, and $e$ depends only on $I$. In particular, $e$ is
	independent of $n$.
	
	\begin{lemma}(See \cite[Corollary 7]{C})\label{c}
		Let $I\lhd A$ be a homogeneous ideal with $\dim A/I = 2$. Then
		$\reg I^{(n)}\leq n\reg (I)$.
	\end{lemma}
	
	The above result of Chandler was generalized in the following sense.
	
	\begin{lemma}(See \cite[Corollary 2.4]{her3})
		Let $I\lhd A$ be a homogeneous ideal with $\dim A/I \leq 2$. Denote the
		maximum degree of the generators of $I$ by $d(I)$. There is a constant $e$ such
		that for all $n > 0$ we have $\reg I^{(n)}\leq nd(I)+e$.
	\end{lemma}
	
	\section{Proof of the observations}
	
	If symbolic powers and ordinary powers coincide, then Huneke's bound is tight.
	We start by presenting some non-trivial examples showing that the desired bound
	is very tight. Historically, these examples are important.
	
	\begin{example}(This has a role in \cite[Page 1801]{her3})
		Let $R=\mathbb{Q}[x,y,z,t]$ and let $I:=(xz,xt^2,y^2z)$. Then $I^{(2)}$ is
		generated in degrees $\leq6=2\times3$. Also, $\Beg(I^{(2)})=2\Beg(I)$.
	\end{example}
	
	\begin{proof}
		The primary decomposition of $I$ is given by
		$$I=(x,y^2)\cap(z,t^2)\cap(x,z).$$
		By definition
		$$I^{(2)}=(x,y^2)^2\cap(z,t^2)^2\cap(x,z)^2
		=(x^2 z^2 , x^2 z t^2, xy^2 z^2 , x ^2t^4, x y^2 zt^2 , y^4 z^2 ).$$
		Thus, $I^{(2)}$ is generated in degrees $\leq6$. Clearly,
		$\Beg(I^{(2)})=4=2\times\Beg(I)$.
	\end{proof}
	
	\begin{example} \label{terai}
		Let $R=\mathbb{Q}[a,b,c,d,e,f]$ and let
		$$I:=(a b c,a b f,a c e,a d e,a d f,
		b c d,b d e,b e f,c d f,c e f).$$
		Then $I^{(2)}$ is generated in degrees $\leq6=2\times3$. Also,
		$\Beg(I^{(2)})=5<6=2\times\Beg(I)$.
	\end{example}
	
	Sturmfels showed that $\reg _1(I^2)=7>6=2\reg _1(I)$. Also, see Discussion
	\ref{lin1}.
	
	\begin{proof}
		This follows from Corollary 1.4. Let us prove it by hand. The method is similar
		to Example 3.1. We leave it to the reader to check that $I^{(2)}$ is generated
		by the following degree 5 elements
		$$\{b c d e f, a c d e f, a b d e f, a b c e f, a b c d f, a b c d e\},$$
		together with the following degree 6 elements
		\[\begin{array}{ll}
			&\{c^2 e^2 f^2 , b c e^2 f^2 ,
			b^2 e^2 f^2 , c^2 d e f^2 , a  b^2 e f^2 , c^2 d^2 f^2,\\
			&a c d^2 f^2 , a^2 d^2 f^2 , a ^2b d f^2 , a^2 b^2 f^2 , b^2 d e^2 f, a c^2 e^2 f,\\
			&a^2 d^2 e f, b c^2 d^2 f, a^2 b^2 c f, b^2 d^2 e^2 , a b d^2 e^2 ,a^2 d^2 e^2 ,\\
			& a^2 cde^2 , a^2 c^2 e^2 , b^2 cd^2 e, a^2 bc^2  e,
			b ^2c^2 d^2 , a b^2 c^2 d, a ^2b^2 c^2\}.
		\end{array}\]
		Thus $I^{(2)}$ is generated in degrees $\leq6=2\times3$. Clearly,
		$\Beg(I^{(2)})=5<6=2\times\Beg(I)$.
	\end{proof}
	
	\begin{proposition}\label{p1}
		Let $I\lhd A$ be a homogeneous ideal such that $\dim  A/I\leq1$. Then
		$I^{(n)}$ is generated in degrees $\leq \max\{n\reg(I)-1,nD\}$ for all $n$.
	\end{proposition}
	
	\begin{proof}
		Let $D$ be such that $I$ is generated in degree $\leq D$. Recall that
		$D\leq \reg(I)$. We note that $I^n$ is generated in degree $\leq Dn$. As
		$\dim(A/I)=1$, in view of Fact \ref{f} one has
		$I^{(n)}/ I^n=\HH^0_{\fm}(A/I^n)$. Now look at the exact sequence
		$$0\lo I^n \stackrel{\rho}\hookrightarrow I^{(n)}\stackrel{\pi}\lo I^{(n)}/ I^n \lo 0.$$
		Suppose $\{f_1,\ldots,f_r\}$ is a homogeneous system of generators for $I^n$.
		Also, suppose $\{\overline{g}_1,\ldots,\overline{g}_s\}$ is a homogeneous system
		of generators for $I^{(n)}/ I^n$, where $g_j\in I^{(n)}$ is defined by
		$\pi(g_j)=\overline{g}_j$. Hence, $\deg(g_j)=\deg(\overline{g}_j)$.
		
		Let us search for a generating set for $I^{(n)}$. To this end, let
		$x\in I^{(n)}$. Hence $\pi(x)=\sum_j s_j \overline{g}_j$ for some $s_j\in A$.
		Thus $x-\sum_j s_jg_j\in\Ker(\pi)=\im(\rho)=I^n$. This says that
		$x-\sum_j s_j g_j=\sum_i r_if_i$ for some $r_i\in A$. Therefore, $\{f_i,g_j\}$
		is a homogeneous generating set for $I^{(n)}$.
		
		Recall that $\deg(f_i)\leq nD$. Fixed $n$, and let $1\leq j \leq s$. It is
		enough to show that
		$$\deg(g_j)\leq n\reg(I)-1.$$
		Keep in mind that $m\geq3$. One has $\depth(A_{\fm})=m> 2$. By
		\cite[Proposition 1.5.15(e)]{BH}, $\grade(\fm, A)\geq 2.$
		Look at
		$$0\lo I^{   n} \lo A \lo A/I^{n}\lo 0.$$
		This induces the following exact sequence:
		$$0\simeq \HH^0_{\fm}( A)\lo \HH^0_{\fm}( A/I^n)\lo \HH^1_{\fm}(   I^n) \lo \HH^1_{\fm}( A)\simeq0.$$
		Thus, $\HH^0_{\fm}( A/I^n)\simeq \HH^1_{\fm}(   I^n)$.
		In view of Lemma \ref{1},
		$$\End(\HH^1_{\fm}(   I^n))+1\leq\reg(I^n)\leq \reg(I)n.$$
		Therefore, $\End(I^{(n)}/ I^n)< \reg(I)n$. By Fact \ref{f},
		$$\deg(g_j)=\deg(\overline{g}_j)\leq\End(I^{(n)}/ I^n)<\reg(I)n,$$
		as claimed.
	\end{proof}
	
	\begin{definition}\label{lin}
		The ideal $ I$ has a linear resolution if its minimal generators all have the
		same degree and the nonzero entries of the matrices of the minimal free
		resolution of $I$ all have degree one.
	\end{definition}
	
	\begin{discussion}\label{lin1}
		In general, powers of ideals with linear resolution need not have linear
		resolutions. The first example of such an ideal was given by Terai, see
		\cite{her4}. It may be worth noting that his example was used in Example
		\ref{terai} for a different purpose.
	\end{discussion}
	
	However, we have:
	
	\begin{corollary}\label{cor}
		Let $I\lhd A$ be an ideal with a linear resolution, generated in degrees
		$\leq D$, such that $\dim A/I\leq1$. Then $I^{(n)}$ is generated in degrees
		$\leq Dn$ for all $n$.
	\end{corollary}
	
	\begin{proof}
		It follows from Definition \ref{lin} that $\reg(I)=D$. Now, Proposition
		\ref{p1} yields the claim.
	\end{proof}
	
	\begin{theorem}\label{tl}
		Let $I\lhd A_4$ be an (radical) unmixed ideal generated in degrees $\leq D$.
		There is an integer $E$ such that $I^{(n)}$ is generated in degrees
		$\leq En$. Suppose in addition that $I$ is of linear resolution. Then $I^{(n)}$ is
		generated in degrees $\leq Dn$.
	\end{theorem}
	
	\begin{proof}
		We note that $I^n$ is generated in degree $\leq Dn$. Suppose first that
		$\Ht(I)=1$. Then $I=(x)$ is principal, because height-one unmixed radical
		ideals over unique factorization domains are principal. In this case
		$I^{(n)}=(x^n)$, because it is a complete intersection.\footnote{In a paper by
			Rees \cite{R}, there is a height-one prime ideal (over a normal domain) such
			that none of its symbolic powers is principal.} In particular, $I^{(n)}$ is
		generated in degrees $\leq Dn$. The case $\Ht(I)=3$ follows by Corollary
		\ref{cor}. Then without loss of generality we may assume that $\Ht(I)=2$.
		Suppose $I^{(n)}$ is generated in degrees $\leq D_n$. Then
		$$D_n\stackrel{\ref{f}}\leq \reg (I^{(n)})\stackrel{\ref{c}}\leq n\reg (I)\stackrel{\ref{lin}}=nD.$$
		The proof  is complete.
	\end{proof}
	
	\begin{theorem}
		Let $I\lhd A_3$ be a homogeneous radical unmixed ideal, generated in degrees $\leq D$
		and of dimension $1$. Then $I^{(n)}$ is generated in degrees
		$< (D+1)n$ for all $n\gg0$.
	\end{theorem}
	
	\begin{proof}
		Keep the proof of Theorem \ref{tl} in mind. Then we may assume
		$\dim (A_3/I)=1$. By the proof of Proposition \ref{p1}, we need to show that
		$\HH^1_{\fm}(   I^n)$ is generated in degrees $<(D+1)n$ for all $n\gg0$. By
		\cite{K}, $\reg(I^n)=a(I)n+b(I)$ for all $n\gg0$. It is well known that
		$a(I)\leq D$, see \cite{K}. For all $n>b(I)$ sufficiently large,
		$$\End(\HH^1_{\fm}(   I^n))+1\leq\reg(I^n)=a(I)n+b(I)\leq Dn+b(I)\leq(D+1)n,$$
		as claimed.
	\end{proof}
	
	\begin{corollary}
		Let $I\lhd A_3$ be any radical unmixed ideal. There is $D\in \mathbb{N}$ such that
		$I^{(n)}$ is generated in degrees $\leq Dn$ for all $n>0$.
	\end{corollary}
	
	\begin{proof}
		Suppose $I$ is generated in degrees $\leq E$ for some $E$. Let $n_0$ be such
		that $I^{(n)}$ is generated in degrees $< (E+1)n$ for all $n>n_0$, see the
		above theorem. Let $\ell_i$ be such that $I^{(n)}$ is generated in degrees
		$< \ell_i$. Now, set $e_i:=\lfloor\frac{\ell_i}{i}\rfloor+1$. Then $I^{(n)}$
		is generated in degrees $< e_in$ for all $n$. Let
		$D:=\sup\{E,e_i:1\leq i\leq n_0\}$. Clearly, $D$ is finite and $I^{(n)}$ is
		generated in degrees $\leq Dn$ for all $n>0$.
	\end{proof}
	
	\begin{corollary}
		Let $I\lhd A_3$ be a homogeneous radical ideal. Then $I^{(n)}$ and $I^n$ have
		the same reflexive hull.
	\end{corollary}
	
	\begin{proof}
		Set $A:=A_3$. Without loss of generality, we may assume that $\Ht(I)=2$. Set
		$(-)^\ast:=\Hom_A(-,A)$. We need to show that
		$(I^{(n)})^{\ast\ast}\simeq (I^{n})^{\ast\ast}$. As
		$I^{(n)}/ I^n=\HH^0_{\fm}(A/I^n)$ is of finite length,
		$\Ext^i_A(I^{(n)}/ I^n,A)=0$ for all $i<3$, because $\depth(A)=3$. Now, we
		apply $(-)^\ast$ to the following exact sequence
		$$0\lo I^n \lo I^{(n)}\lo I^{(n)}/ I^n \lo 0,$$
		to observe that $(I^{(n)})^{\ast}\simeq (I^{n})^{\ast}$. From this we get the
		claim.
	\end{proof}
	
	To prove Observation C we need:
	
	\begin{fact}\label{fg} (See \cite[Theorem 3.2]{her2})
		Let $I$ be a monomial ideal in a polynomial ring over a field. Then the
		corresponding symbolic Rees algebra is finitely generated.
	\end{fact}
	
	Here, we present the proof of Observation C.
	
	\begin{proposition}
		Let $I$ be any ideal such that the corresponding symbolic Rees algebra is
		finitely generated (e.g. $I$ is monomial). There is $D\in \mathbb{N}$ such that
		$I^{(n)}$ is generated in degrees $\leq Dn$ for all $n>0$.
	\end{proposition}
	
	\begin{proof}(Suppose $I$ is monomial. Then
		$\mathcal{R}:=\bigoplus_{i\geq 0}I^{(i)}$ is finitely generated, see Fact
		\ref{fg}.) The finiteness of $\mathcal{R}$ gives an integer $\ell$ such that
		$\mathcal{R}=R[I^{(i)}:i\leq \ell]$. Let $e_i$ be such that $I^{(i)}$ is
		generated in degree less than or equal to $e_i$ for all $i\in \mathbb{N}$. Set
		$d_i:=\lfloor\frac{e_i}{i}\rfloor+1$ for all $i\leq \ell$. The notation $D$
		stands for $\max\{d_i:\textit{for all }i\leq \ell\}$. We note that $D$ is
		finite. Let $n$ be any integer. Then
		$$I^{(n)}=\sum I^{(j_1)}\ldots I^{(j_{i_n})},\ \  \text{where} \ \ j_1+\ldots +j_{i_n}=n
		\ \ \text{and}  \ \ 1\leq j_i\leq\ell\quad \text{for all} \ \ 1\leq i\leq n.$$
		This implies that
		\[\begin{array}{ll}
			e_n&\leq e_{j_1}+\ldots+e_{j_{i_n}}\\
			&\leq j_1d_{j_1}+\ldots+j_{i_n}d_{j_{i_n}}\\
			&\leq j_1D+\ldots+j_{i_n}D\\&=D(j_1+\ldots+j_{i_n})\\
			&=Dn,
		\end{array}\]
		as claimed.
	\end{proof}
	
	\begin{corollary}
		Let $I=(f_1,\ldots,f_t)$ be a monomial ideal. Set
		$E:=\deg f_1+\ldots+\deg f_t$. Then $I^{(n)}$ is generated in degrees
		$\leq En$ for all $n>0$.
	\end{corollary}
	
	\begin{proof}
		We may assume $I\neq 0$. Thus, $\Ht(I)\geq 1$. Let $f$ be the least common
		multiple of the generating monomials of $I$. In view of
		\cite[Theorem 2.9]{her3}, for all $n > 0$,
		$$\reg(I^{(n)})\leq (\deg f)n - \Ht( I) + 1 \quad(\ast)$$
		The notation $e_n$ stands for the maximal degree of the generators of
		$I^{(n)}$. Due to Fact \ref{f} we have $e_n\leq\reg(I^{(n)})$. Putting this
		together with $(\ast)$ we observe that
		$$e_n\leq\reg(I^{(n)})\leq(\deg f)n - \Ht( I) + 1\leq(\deg f)n $$
		for all $n>0$. It is clear that
		$\deg f\leq \deg(\prod_i f_i)=\sum_i\deg f_i=E$.
	\end{proof}
	
	One may like to deal with the following sharper bound.
	
	\begin{corollary}\label{hhh}
		Let $I$ be a monomial ideal and let $f$ be the least common multiple of the
		generating monomials of $I$. Then $I^{(n)}$ is generated in degrees
		$\leq \deg(f)n$ for all $n>0$.
	\end{corollary}
	
	\section{Symbolic powers of high degree}
	
	\subsection{Proof of Example \ref{introduction}}
	
	We start with a computation from Macaulay2.
	
	\begin{lemma}\label{hop}
		Let $A:=\mathbb{Q}[x,y,z,t,a,b]$ and let
		$M:=\left(x(x-y)ya,(x-y)ztb,yz(xa-tb)\right)$. Set
		$f:=xy(x-y)ztab(ya-tb)$. The following holds:
		\begin{enumerate}
			\item[i)]  $(M^2:_Af)=(x,y,z)$,
			\item[ii)] $M$ is a radical ideal and all the associated primes of $M$ have
			height $2$. In fact
			$$\Ass(M)=\{(b, a), (b, x), (y, x), (z, x), (t, x), (b, y),  (z,
			y), (t, y),  (a, t),  (a, z),  (z, x - y),  (x - y, ya - tb)\}$$
		\end{enumerate}
	\end{lemma}
	
	\begin{proof}
		The following is the Macaulay2 session.
		\begin{verbatim}
			i1 : R=QQ[x,y,z,t,a,b]
			o1 = R
			o1 : PolynomialRing
			i2 : M=ideal(x*(x-y)*y*a,(x-y)*z*t*b,y*z*(x*a-t*b))
			o2 = ideal (x(x-y)ya, (x-y)ztb, yz(xa-tb))
			o2 : Ideal of R
			i3 : Q = quotient(J*J,x*y*(x-y)*z*t*a*b*(y*a-t*b))
			o3 = ideal (z, y, x)
			o3 : Ideal of R
			i4 : associatedPrimes M
			o4 = {ideal (b, a), ideal (b, x), ideal (y, x), ideal (z, x), ideal (t, x),
				ideal (b, y), ideal (z, y), ideal (t, y), ideal (a, t), ideal (a, z),
				ideal (z, x - y), ideal (x - y, y*a - t*b)}
			o4 : List
		\end{verbatim}
		It is easy to see that $M$ is radical. These prove the items i) and ii).
	\end{proof}
	
	\begin{lemma}
		Adopt the above notation. Then $M^{(2)}= M^{2}+(f)$.
	\end{lemma}
	
	\begin{proof}[Sketch of Proof]
		Denote the set of all associated prime ideals of $M$ by
		$\{\fp_i:1\leq i\leq 12\}$ as listed in Lemma \ref{hop}. Revisiting Lemma
		\ref{hop} we see that $M$ is radical. Thus $M=\bigcap_{1\leq i\leq 12}\fp_i$.
		By definition,
		$$M^{(2)}=\bigcap_{1\leq i\leq 12}\fp_i^2.$$
		For simplicity, we relabel $A:=\fp_1^2$, $B:=\fp_2^2$ and so on, and finally
		$L:=\fp_{12}^2$. In order to compute this intersection we use Macaulay2.
		\begin{verbatim}
			i5 : Y=intersect(A,B,C,D,E,F,G,H,I,J,K,L)
			o5 : Ideal of R
			i6 : N=ideal(x*y*(x-y)*z*t*a*b*(y*a-t*b))
			o7 : Ideal of R
			i8 : M*M+N==Y
			o8 = true
		\end{verbatim}
		The output term ``\emph{true}'' means that the claim
		``$M^{(2)}= M^{2}+(f)$'' is true.
	\end{proof}
	
	Now we are ready to present:
	
	\begin{example}
		Let $A:=\mathbb{Q}[x,y,z,t,a,b]$ and
		$J:=(x(x-y)ya,(x-y)ztb,yz(xa-tb))$. Then $J$ is generated by degree-four
		elements and $J^{(2)}$ has a minimal generator of degree $9$.
	\end{example}
	
	\begin{proof}
		Let $f$ be as in Lemma \ref{hop}. In the light of Lemma \ref{hop},
		$(J^2:_Af)=(x,y,z)$. Thus, $f\notin J^2$. This means that $f$ is a minimal
		generator of $J^{(2)}$. Since $\deg(f)=9$ we get the claim.
	\end{proof}
	
	A somewhat simpler example (in dimension $7$) is the following.
	
	\begin{example}
		Let $A:=\mathbb{Q}[x,y,z,a,b,c,d]$ and
		$I:=(x y a b,x z cd,y z(a b-c d))$. Then $I$ is radical, binomial,
		Cohen--Macaulay of height $2$. Clearly, $I$ is generated in degree $4$. But
		$I^{(2)}$ has a minimal generator of degree $9$.
	\end{example}
	
	\begin{proof}
		Let $f=x yz a bc d(a c-b d)$. By using Macaulay2, we have
		$(I^2:_Af)=(x,y,z).$ The same computation shows that $I$ is radical, binomial,
		Cohen--Macaulay of height $2$. Clearly, $I$ is generated in degree $4$. But
		$I^{(2)}=I^2+(f)$, and $f$ is a minimal generator of $I^{(2)}$ of degree $9$.
	\end{proof}
	
	\subsection{An exact family of symbolic squares}
	
	Fix an integer $r\geq1$ and let
	\[
	S_r=K[a,b,c,a_1,\ldots,a_r,b_1,\ldots,b_r,c_1,\ldots,c_r].
	\]
	Write $A=\prod_{i=1}^{r}a_i$, $B=\prod_{i=1}^{r}b_i$ and
	$C=\prod_{i=1}^{r}c_i$, and put
$$
		I_r=(abc,\;bcA,\;acB,\;abC),
		\qquad F_r=abcABC.
$$
	The graph consisting of the triangle on $a,b,c$ and $r$ leaves adjacent
	to each triangle vertex appears in \cite[Example~6.5]{HaTrung}; it is
	particularly convenient for a completely explicit symbolic-square
	calculation. The following strengthens the elementary numerical
	calculation for this family by recording generators and a defect module.
	
	\begin{example}\label{thm:family}
	The following assertions are valid.
	
		\begin{enumerate}
	\item[1)]  For every $r\geq1$ the ideal $I_r$ is squarefree and radical, with
	$d(I_r)=r+2$.
 
	\item[2)]  	$I_r^{(2)} =I_r^2+(F_r)$
	\item[3)] $(I_r^2:F_r) =(a,b,c)$
		\item[4)]  $	I_r^{(2)}/I_r^2 \cong
		S_r/(a,b,c)(-3r-3)$
	
	\item[5)]  $d(I_r^{(2)}) =3r+3$
	\item[6)] In particular, for $r\geq2$,
$ 
		d(I_r^{(2)})-2d(I_r)=r-1>0.
$ 
	The ideal $I_r^{(2)}$ has exactly eleven minimal monomial generators:
	the ten pairwise products, with repetition, of the four generators of
	$I_r$, and the additional monomial $F_r$.
\end{enumerate}	
	\end{example}
	
	\begin{proof}
		Let $G_r$ be the indicated graph. Its minimal vertex covers are
		$\{a,b,c\}$ and the three sets obtained by taking two triangle
		vertices together with all $r$ leaves at the omitted vertex. Hence
$$
			I_r=\bigcap_{\{u,v\}\in E(G_r)}(u,v).
$$
		Every displayed edge prime has height two and none contains another;
		this is the irredundant minimal-prime decomposition. In particular
		$I_r$ is radical and its four displayed monomial generators are minimal.
		Their degrees are $3,r+2,r+2,r+2$, so $d(I_r)=r+2$.
		By $
		I_r=\bigcap_{\{u,v\}\in E(G_r)}(u,v) 
		$,
$$
			I_r^{(2)}=\bigcap_{\{u,v\}\in E(G_r)}(u,v)^2.
$$
		
		Consider a monomial $M\in I_r^{(2)}$. Because membership in
		$(u,v)^2$ is the condition that the sum of the exponents on $u,v$ is
		at least two, the exponents on $a,b,c$ can be lowered, if necessary,
		to numbers $\alpha,\beta,\gamma\in\{0,1,2\}$ without leaving the
		intersection. These must satisfy
$$(\ast):\quad\dfrac{num}{den}
			\alpha+\beta\geq2,\qquad
			\alpha+\gamma\geq2,\qquad
			\beta+\gamma\geq2.
$$
		Once $\alpha,\beta,\gamma$ are fixed, the smallest admissible
		exponent on each $a_i$ is $2-\alpha$, and similarly it is
		$2-\beta$ on each $b_i$ and $2-\gamma$ on each $c_i$.
		Consequently every monomial in the intersection is divisible by one
		of the monomials
$$(+)\quad
			M_{\alpha,\beta,\gamma}
			=a^\alpha b^\beta c^\gamma
			A^{2-\alpha}B^{2-\beta}C^{2-\gamma},
$$
		where $(\ast)$ holds. Conversely all these
		monomials belong to the intersection.
		
		The admissible triples, up to permutation, are exactly
		\[
		(2,2,2),\quad(2,2,1),\quad(2,2,0),\quad
		(2,1,1),\quad(1,1,1).
		\]
		The first four types, with their permutations, are precisely the ten
		pairwise products of $abc,bcA,acB,abC$. The last type is $F_r$.
		This proves (2).
		
		These eleven monomials are pairwise incomparable under divisibility.
		Indeed, if $M_{\alpha,\beta,\gamma}$ divides
		$M_{\alpha',\beta',\gamma'}$, the exponents of $a,b,c$ give
		$\alpha\leq\alpha'$, $\beta\leq\beta'$, $\gamma\leq\gamma'$;
		comparison of the exponents on any $a_i,b_i,c_i$ gives the opposite
		inequalities. Thus the triples must coincide. This proves the
		assertion about minimal generators.
		
		Every product of two of the four generators of $I_r$ has exponent at
		least two on some one of $a,b,c$, while $F_r$ has exponent one on each.
		It follows that
		$(I_r^2:F_r)\subseteq(a,b,c)$. Conversely,
		\[
		aF_r\in(acB)(abC)S_r,\quad
		bF_r\in(bcA)(abC)S_r,\quad
		cF_r\in(bcA)(acB)S_r.
		\]
		Hence $(a,b,c)\subseteq(I_r^2:F_r)$, proving (3)
		The isomorphism (4) follows from the map
		$S_r(-\deg F_r)\to I_r^{(2)}/I_r^2$ sending $1$ to the class of
		$F_r$.
		
		Finally, $(+)$ has degree
		\[
		\deg M_{\alpha,\beta,\gamma}
		=6r+(1-r)(\alpha+\beta+\gamma).
		\]
		The minimum of $\alpha+\beta+\gamma$ among admissible triples is
		$3$, attained at $(1,1,1)$. Since $r\geq1$, the largest possible
		degree is $6r+3(1-r)=3r+3$. This proves (5) and
(6).
	\end{proof}
	
	\section{A polynomial bound for  homogeneous ideals}
	
	We now address the polynomial-bound part of Problem 1.5. 
		\subsection{A polynomial bound for arbitrary homogeneous ideals}
	The argument rests
	on a single saturation element that works simultaneously for all symbolic
	exponents.
	
	\begin{lemma}[Uniform one-element saturation]\label{lem:uniform-sat}
		Let $I\subsetneq S$ be any nonzero homogeneous ideal, with symbolic
		powers. There exists a nonzero polynomial $h\in S$ such that
		$$(\ast):\quad
		I^{(n)}=I^n:h^\infty
		:=\bigcup_{e\geq0}(I^n:h^e)
		\qquad\text{for every }n\geq1.
		$$
		The same conclusion holds if symbolic powers are instead defined with
		respect to the minimal primes of $I$.
	\end{lemma}
	
	\begin{proof}
		Let $\mathcal P=\Ass_S(S/I)$, and put
$
		W=S\setminus\bigcup_{\pp\in\mathcal P}\pp.
$
		Prime avoidance applied to the colon ideals $(I^n:f)$ shows that
		$$(+):\quad
		I^{(n)}=I^nS_W\cap S.
		$$
		For completeness, $f$ belongs to the right-hand side if and only if
		$(I^n:f)$ contains an element outside all primes in $\mathcal P$.
		By prime avoidance, this is equivalent to $(I^n:f)$ not being contained
		in any one $\pp\in\mathcal P$, which is equivalent to
		$f\in I^nS_\pp$ for every such $\pp$.
		
		By Brodmann's theorem \cite{Brodmann}, the set
$
		\mathcal A=\bigcup_{n\geq1}\Ass_S(S/I^n)
$
		is finite. For each $\mathfrak q\in\mathcal A$ with
		$\mathfrak q\cap W\neq\varnothing$, choose
		$h_{\mathfrak q}\in\mathfrak q\cap W$, and let $h$ be the product of
		these finitely many elements (set $h=1$ if there are none). Then
		$h\in W$. Furthermore, for every $\mathfrak q\in\mathcal A$,
		\[
		h\in\mathfrak q\quad\Longleftrightarrow\quad
		\mathfrak q\cap W\neq\varnothing.
		\]
		Indeed, the forward implication follows from $h\in W$, and the
		reverse implication holds by construction.
		
		For a primary decomposition of $I^n$, saturation by $h$ deletes
		exactly those primary components whose associated primes contain
		$h$. Localization at $W$, followed by contraction, deletes exactly
		those primary components whose associated primes meet $W$.
		By the displayed equivalence, the surviving components are the same.
		Thus $(\ast)$ follows from $(+)$.
		For the minimal-prime convention one repeats the argument with
		$\mathcal P=\Min(I)$.
	\end{proof}
Here, is a general polynomial generator-degree bound.
	\begin{proposition}\label{thm:polynomial}
		Let $I\subsetneq A=K[x_1,\ldots,x_m]$ be a nonzero homogeneous ideal.
		There exists an integer $C_I\geq1$ such that, for every $n\geq1$,
$ d(I^{(n)})\leq
				2(\frac{C_I^2n^2}{2}+C_In)^{2^m}. 
$
		In particular, $d(I^{(n)})$ is bounded by a polynomial in $n$ of
		degree $2^{m+1}$. The statement holds under either symbolic-power
		convention described in Lemma~\ref{lem:uniform-sat}.
	\end{proposition}
	
	\begin{proof}
		Choose $h$ as in Lemma~\ref{lem:uniform-sat} and adjoin one
		indeterminate $T$. The standard elimination identity says
$
			I^n:h^\infty=\bigl(I^n,1-Th\bigr)A[T]\cap A.
$
		If $D=d(I)$, the ideal
		$J_n=(I^n,1-Th)\subseteq K[x_1,\ldots,x_m,T]$ admits a generating
		set of total degrees at most
		\[
		\max\{Dn,\,\deg h+1\}\leq C_In,
		\qquad C_I:=\max\{D,\deg h+1,1\}.
		\]
		Choose any monomial order eliminating $T$. Dub\'e's universal
		Gr\"obner-basis degree bound \cite{Dube}, applied in $m+1$
		variables, gives a Gr\"obner basis of $J_n$ with all total degrees
		at most
$
		2(\frac{C_I^2n^2}{2}+C_In)^{2^m}.
$
		The elimination theorem shows that the elements of this basis
		belonging to $S$ generate $J_n\cap A=I^{(n)}$. Because this
		intersection is homogeneous, taking homogeneous components if needed
		gives homogeneous generators with no greater degree. This proves
the desired claim.
	\end{proof}
	
	\begin{remark}
		The bound in Theorem~\ref{thm:polynomial} is extremely coarse. Its
		substance is polynomial rather than linear growth, not the numerical
		exponent $2^{m+1}$. It does \emph{not} settle the question of whether
		$d(I^{(n)})\leq C_In$ for every homogeneous ideal $I$. For monomial
		ideals a linear bound is known, for instance from finite generation of
		the symbolic Rees algebra.
	\end{remark}
	
	\subsection{An asymptotically sharp estimate for some primes}
	
	For a homogeneous prime ideal $P\subset S$, the next result approximates the
	requested bound to within a constant independent of $n$.
	
	\begin{proposition}\label{prop:lci}
		Let $P\neq\mm$ be a homogeneous prime ideal of
		$A=K[x_1,\ldots,x_m]$, generated in degrees at most $D$.
		Suppose $P$ is locally a complete intersection at every prime of
		$A$ other than $\mm$. Then there exists an integer $B_P\geq0$ such
		that
$d(P^{(n)})\leq Dn+B_P $ for $n\geq1$.
		Consequently,
$
		\limsup_{n\to\infty}\frac{d(P^{(n)})}{n}\leq D.
$
	\end{proposition}
	
	\begin{proof}
		For every prime $\mathfrak q\neq\mm$ containing $P$, the ideal
		$PA_{\mathfrak q}$ is a complete-intersection prime.
		The powers of a complete-intersection prime are primary, so
		$P^nA_{\mathfrak q}$ is $PA_{\mathfrak q}$-primary for every $n$.
		It follows that $P^n$ has no embedded associated primes except,
		possibly, $\mm$. Thus
$P^{(n)}=(P^n:\mm^\infty).$
		Saturation with respect to $\mm$ does not increase
		Castelnuovo--Mumford regularity. Therefore
		\[
		d(P^{(n)})\leq\reg(P^{(n)})\leq\reg(P^n).
		\]
		The asymptotic regularity theorem of
		Cutkosky--Herzog--Trung \cite{Cu} and Kodiyalam \cite{K}
		gives $\reg(P^n)=en+b$ for all sufficiently large $n$, where
		$e\leq d(P)\leq D$. Enlarging an integer $B_P\geq0$ handles the
		finitely many smaller values of $n$ and gives {Proposition} \ref{prop:lci}.
	\end{proof}
	
	\begin{remark}
		The proposition applies, for example, when the projective scheme
		$\operatorname{Proj}(S/P)$ is smooth over $K$. Even in this class the
		proposition does not prove the exact bound $d(P^{(n)})\leq Dn$;
		the additive constant $B_P$ cannot simply be discarded.
	\end{remark}
	
	\begin{acknowledgement}
		I thank Hop D. Nguyen for finding a gap, for filling the gap, and for several
		simplifications. We checked some of the examples with the help of Macaulay2.
	\end{acknowledgement}
	


\begin{thebibliography}{99}
		
		\bibitem{Brodmann}
		M. Brodmann, \emph{Asymptotic stability of $\operatorname{Ass}(M/I^nM)$},
		Proc. Amer. Math. Soc. \textbf{74} (1979), 16--18.
		
		\bibitem{BH}
		W. Bruns and J. Herzog, \emph{Cohen-Macaulay rings}, Cambridge University
		Press, \textbf{39}, Cambridge, (1998).
		
		\bibitem{C}
		Karen A. Chandler, \emph{Regularity of the powers of an ideal},
		Comm. Algebra \textbf{25} (1997), 3773--3776.
		
		\bibitem{her4}
		A. Conca, \emph{Hilbert function and resolution of the powers of the ideal of
			the rational normal curve}, J. Pure Appl. Algebra \textbf{152} (2000) 65--74.
		
		\bibitem{Cu}
		S. Cutkosky, J. Herzog, and N.V. Trung, \emph{Asymptotic behavior of the
			Castelnuovo--Mumford regularity}, Compositio Math. \textbf{118} (1999),
		243--261.
		
		\bibitem{Dube}
		T. W. Dub\'e, \emph{The structure of polynomial ideals and Gr\"obner bases},
		SIAM J. Comput. \textbf{19} (1990), 750--773.
		
		\bibitem{ger}
		A.V. Geramita, A. Gimigliano and Y. Pitteloud, \emph{Graded Betti numbers of
			some embedded rational n-folds}, Math. Ann. \textbf{301} (1995), 363--380.
		
		\bibitem{mac}
		D. Grayson and M. Stillman, \emph{Macaulay2: a software system for research in
			algebraic geometry}, Available at http://www.math.uiuc.edu/Macaulay2/.
		
		\bibitem{HaTrung}
		H. T. H\`a and N. V. Trung,
		\emph{Membership criteria and containments of powers of monomial ideals},
		Acta Math. Vietnam. \textbf{44} (2019), 117--139.
		
		\bibitem{h1}
		C. Huneke, \emph{Open problems on powers of ideals}, Notes from a workshop on
		Integral Closure, Multiplier Ideals and Cores, AIM, December 2006.
		
		\bibitem{her3}
		J. Herzog, L.T. Hoa and N.V. Trung, \emph{Asymptotic linear bounds for the
			Castelnuovo-Mumford regularity}, Trans. Amer. Math. Soc. \textbf{354} (2002),
		1793--1809.
		
		\bibitem{her2}
		J. Herzog, T. Hibi, and N.V. Trung, \emph{Symbolic powers of monomial ideals
			and vertex cover algebras}, Adv. Math. \textbf{210} (2007), 304--322.
		
		\bibitem{K}
		V. Kodiyalam, \emph{Asymptotic behavior of Castelnuovo-Mumford regularity},
		Proc. AMS \textbf{128} (2000), 407--411.
		
		\bibitem{R}
		D. Rees, \emph{On a problem of Zariski}, Illinois J. Math. \textbf{2}, (1958)
		145--149.
		
		\bibitem{p}
		P. Roberts, \emph{A prime ideal in a polynomial ring whose symbolic blow-up is
			not Noetherian}, Proc. AMS. \textbf{94} (1985), 589--592.
		
		\bibitem{st}
		B. Sturmfels, \emph{Four counterexamples in combinatorial algebraic geometry},
		J. Algebra \textbf{230} (2000), 282--294.
		
	\end{thebibliography}
\end{document}